\documentclass[11pt]{article}
\usepackage{amsmath,amssymb,amsbsy,amsfonts,amsthm,latexsym,
            amsopn,amstext,amsxtra,euscript,amscd,amsthm}

\newtheorem{lem}{Lemma}
\newtheorem{lemma}[lem]{Lemma}

\newtheorem{thm}{Theorem}
\newtheorem{theorem}[thm]{Theorem}

\def\\{\cr}
\def\({\left(}
\def\){\right)}
\def\[{\left[}
\def\]{\right]}
\def\<{\langle}
\def\>{\rangle}

\begin{document}

\title{On Carmichael numbers of the form $2^np^m+1$}

\date{\today}

\pagenumbering{arabic}

\author{Florian~Luca\footnote{Mathematics Division, Stellenbosch University, South Africa and Max Planck Institute for Software Systems, Saarbr\"ucken, Germany, e-mail: fluca@sun.ac.za}}

\maketitle

\begin{abstract}
Here, we show that if $m\ge 5$ is fixed and odd, then there are only finitely many Carmichael numbers of the form $2^np^m+1$ for positive integers $n$ and prime $p$. 
\end{abstract}

\section{Introduction}

A positive integer $N$ is Carmichael if it composite and $a^N\equiv a\pmod N$ holds for all integers $a$. It is known that there are infinitely many Carmichael numbers. They are all odd, squarefree and are characterized by  the property that $p-1\mid N-1$ for all prime factors $p$ of $N$. Every such number is of the form $2^nk+1$ for some positive integer $n$ and odd integer $k$. Let
$$
{\mathcal K}:=\{k~{\text{\rm odd}}: 2^nk+1~{\text{\rm is~Carmichael~for~some~positive~integer}}~n\},
$$
The set ${\mathcal K}$ has been investigated in a few papers. For example, in \cite{CLP}, it is shown that the smallest element in ${\mathcal K}$ is $27$. It comes from the Carmichael number $1729=2^6\times 3^3+1$ which also happens to be the famous Ramanujan taxicab number. In \cite{AL}, it was shown that ${\mathcal K}$ does not contain primes, in \cite{LR} it was shown that ${\mathcal K}$ does not contain squares of primes and in \cite{L1} 
it was shown that the only prime $p$ such that $p^3\in {\mathcal K}$ is $p=3$.  
This time we are studying numbers of the form $p^m$ in ${\mathcal K}$ for some $m\ge 4$ and prove the following theorem.

\begin{theorem}
\label{thm:1}
If $m\ge 5$ is odd then there only finitely many primes $p$ such that $p^m\in {\mathcal K}$. 
\end{theorem}

Our proof uses a result of Corvaja and Zannier from \cite{CZ} whose proof is based on the Schmidt's Subspace Theorem and therefore the proof is not effective. 
 
\section{The proof of Theorem \ref{thm:1}}

In Theorem 1 in \cite{CLP}, it was shown that if $k\ge 3$ is odd and is such that $N=2^nk+1$ is Carmichael, then 
$$
n<2^{2\times 10^7 \tau(k)^2 (\log k)^2 \omega(k)},
$$
where $\omega(k),~\tau(k)$ are the number of distinct prime factors of $k$ and the total number of positive divisors of $k$, respectively. In particular, if $k=p^m$ and $m$ is fixed, then 
$$
n<2^{2\times 10^7 (m+1)^2 m^2(\log p)^2}.
$$
Thus, if we fix $m$ and show that $p$ can take only finitely many values it will follow that there are only finitely many Carmichael numbers of the form $2^np^m+1$ for some prime $p$. This is what we will prove when $m$ is odd.

\subsection{Bounding $N$}

We follow the method of proof of Theorem 1 in \cite{CLP}. We  assume $m\ge 5$ and odd, $p>m$ is prime and let $N:=2^np^m+1$. Assume that $N$ is Carmichael and write 
\begin{equation}
\label{eq:0}
N=\prod_{i=1}^k (2^{a_i}p^{b_i}+1),\qquad 1\le a_i\le n,~0\le b_i\le m,
\end{equation}
where 
$$
q_{i}:=2^{a_i}p^{b_i}+1
$$
are distinct primes. It is well-known that $k=\omega(N)\ge 3$.  
The primes $q_i$ for which $b_i=0$, so $q_i=2^{a_i}+1$ are Fermat primes, therefore $a_{i}=2^{\beta_i}$ for all some positive integer $\beta_i$. Let us denote by $\alpha:=\max\{\beta_i: b_i=0\}$ the maximal exponent of a Fermat prime participating in the factorization \eqref{eq:0} of $N$. Lemma 2 in \cite{CLP} shows that $2^{2^{\alpha}}+1<p^{2m}$. In particular,
\begin{equation}
\label{eq:N0}
N_0:=\prod_{\substack{1\le i\le k\\ b_i=0}} q_i\le \prod_{j=0}^{\alpha} (2^{2^j}-1)=2^{2^{\alpha+1}}-1<(p^{2m})^2=p^{4m}.
\end{equation}
Next we show that $N/N_0>1$. Indeed, for if not, we have 
$$
N=N_0=\prod_{i=1}^k (2^{a_i}+1),
$$
where we may assume that $1\le a_1<a_2<\cdots<a_k$ and all $a_i$ are powers of $2$ for $i=1,\ldots,k$. Since $k\ge 3$, we have $a_1+1\le a_2\le n$, so we may reduce the above equation modulo $2^{a_1+1}$ to get that
$$
1\equiv 2^{a_1}+1\pmod {2^{a_1+1}},
$$
which is a contradiction. Thus, putting 
\begin{equation}
\label{eq:N1}
N_1:=N/N_0,
\end{equation}
we have that $N_1>1$.  Next let $i$ be such that $b_i\ge 1$ and let $q_{i}=2^{a_i}p^{b_i}+1$ be a prime factor of $N$. Assume first that $2^np^m$ and $2^{a_i}p^{b_i}$ are multiplicatively dependent. It is easy to see that this happens if and only if there exists an integer $\rho>1$ and an integer $t\ge 2$ such that 
$$
2^{n}p^m=\rho^t\qquad {\text{\rm and}}\qquad 2^{a_i}p^{b_i}=\rho^{\lambda}\qquad {\text{\rm for~some}}\qquad 1\le \lambda<t.
$$ 
The condition that $q_i\mid N$ implies
$$
\rho^{\lambda}+1\mid \rho^t+1\quad {\text{\rm which leads to}}\quad \lambda\mid t\quad {\text{\rm and}}\quad t/\lambda\equiv 1\pmod 2,
$$
and the condition that $q_i=\rho^{\lambda}+1$ is prime implies that $\lambda$ is a power of $2$. Interpreting this in terms of the exponent of the prime $p$ in the factorization of $\rho$, we get that $\lambda$ is the largest power of $2$ dividing $m$. 
Since $m$ is odd, we get that $\lambda=1$, so $t=m$. Thus, $n=mu$ for some integer $u$, $(a_i,b_i)=(u,1)$, $\rho=2^up$ and  
$$
2^np^m+1=(2^{u}p+1)((2^{u}p)^{m-1}-(2^up)^{m-2}+\cdots-2^u p+1).
$$
Let $q$ be any prime factor of the second factor above. We then get that $(2^up)^{m}\equiv -1\pmod q$, so the order of $2^up$ modulo $q$ is even and a divisor of $2m$. However, since $q$ is one of the $q_i$'s it follows that this order is a divisor of $q_i-1\mid 2^np^m$.  Since $p>m$, this order must be exactly $2$. In particular, $2^{u}p\equiv -1\pmod q$, showing that $q$ divides $2^{u}p+1$. Thus, 
$q$ divides both $2^up+1$ and $((2^{u}p)^m+1)/(2^up+1)$, so $q^2\mid N$, a contradiction. 

Let us record what we have proved.

\begin{lemma}
Assume $m\ge 5$ is odd, $p>m$ is prime and $N=2^np^m+1$ is Carmichael. Then with $N_0$ and $N_1$ given by \eqref{eq:N0} and \eqref{eq:N1}, respectively, we have that $N_1>1$ and if $q_i=2^{a_i}p^{b_i}+1$ is any prime factor of $N$, then 
$2^np^m$ and $2^{a_i}p^{b_i}$ are multiplicatively independent. 
\end{lemma}
Let 
$$b:=\min\{b_i: q_i\mid N\}\quad {\text{\rm and}}\quad {I}(c):=\{i: q_i=2^{a_i}p^b+1\},\quad 0\le c\le m.
$$ 
Thus, $I(c)$ is exactly the subset of indices $i\in \{1,\ldots,k\}$ such that 
$$
\nu_{p}(q_i-1)=c.
$$ 
Here and in what follows for a prime $q$ and a nonzero integer $m$ we write $\nu_q(m)$ for the exact exponent of $q$ in the factorization of $m$. 
We prove the following lemma.

\begin{lemma}
\label{lem:10}
If $b=0$, then 
\begin{equation}
\label{eq:101}
p\mid N_0-1. 
\end{equation}
If $b>0$, then 
\begin{equation}
\label{eq:102}
p\mid \sum_{i\in I(b)} 2^{a_i},\qquad {\text{\rm or}}\qquad p\mid 2^n-\sum_{i\in I(b)} 2^{a_i},
\end{equation}
according to whether $b\le m-1$ or $b=m$, respectively.  In the latter case, the right--hand side of \eqref{eq:102} is nonzero. 
In all cases, $\#I(b)\ge 2$.
\end{lemma}

\begin{proof}
If $b=0$, then $N_0>1$. Further, reducing the equation $N=N_0N_1$ modulo $p$ and noting that $N_1\equiv 1\pmod p$, we get that $p\mid N_0-1$. If $b\ge 1$, then $N_0=1$ and $q_j=2^{a_j}p^{b_j}+1\equiv 0\pmod {p^{b+1}}$ 
for all $j\not\in I(b)$. Thus, reducing equation \eqref{eq:0} modulo $p^{b+1}$, we get that 
$$
N\equiv \prod_{i\in I(b)} (2^{a_i}p^{b}+1)\pmod {p^{b+1}}\equiv 1+p^b(\sum_{i\in I(b)} 2^{a_i})\pmod {p^{b+1}}.
$$
If $b+1\le m$, we get that the left--hand side above is congruent to $1$ modulo $p^{b+1}$, so we get that $p\mid \sum_{i\in I(b)} 2^{a_i}$. Otherwise, we must have that $b=m$, $N\equiv 2^np^b+1\pmod {p^{b+1}}$, and we get that 
$$
p\mid 2^n-(\sum_{i\in I(b)} 2^{a_i}).
$$
If $b=0$ or $1\le b\le m-1$, we get that $\#I(b)\ge 2$, otherwise divisibilities \eqref{eq:101} and \eqref{eq:102}-left above would lead to the conclusion that with $I(b)=\{i\}$ for some we have $p\mid 2^{a_i}$, which is impossible. 
Finally, when $b=m$, then $\#I(b)=k\ge 3$ (all prime factors of $N$ are of the form $q_i=2^{a_i}p^b+1$ for some positive integers $a_i$ which are distinct), and it follows by the unicity of the base $2$ representation of a positive integer that the right--hand side of \eqref{eq:102} is nonzero. 
\end{proof}

We next analyse the prime factors of $N_1$.  Let $q_{i}$ be any prime factor of $N_1$. We then have
\begin{eqnarray*}
2^{a_i}p^{b_i} & \equiv &  -1\pmod {q_i},\\ 
2^np^m & \equiv & -1\pmod {q_i}.
\end{eqnarray*}
Let $d_i:=\gcd(m,b_i)$ and write $m:=d_i m_i$ and $b_i:=d_ic_i$. Raising the first congruence to $m_i$ and the second to $c_i$ and dividing them side by side, we get
\begin{equation}
\label{eq:lambda}
2^{\lambda_{i}}\equiv \varepsilon_{i}\pmod {q_i},
\end{equation}
where 
$$
(\lambda_{i},\varepsilon_{i}):=(nc_i-a_i m_i,(-1)^{c_i-m_i}),\qquad 1\le i\le k,\quad q_i\mid N_1.
$$ 
Note that $\lambda_{i}\ne 0$ since $2^{a_i}p^{b_i}$ and $2^np^m$ are multiplicatively independent. Since $a_i\in [1,n]$ and $b_i\in [1,m]$, we conclude that $|\lambda_{i}|\le mn$ in all cases. Further, 
since $\varepsilon_{i}\in \{\pm 1\}$, we get that $2^{2\lambda_{i}}\equiv 1\pmod {q_i}$, so in all cases $2|\lambda_{i}|\le 2mn$ is a positive integer which is a multiple of the order of $2$ modulo $q_i$, which in turn is a divisor of $2^np^3$. 
Thus, letting $u_s$ be the largest positive integer which is of the form 
$$
u_s=2^{\gamma_s} p^s\quad {\text{\rm for}}\quad s=0,1,2,\ldots,m,
$$
and which is at most $2mn$, we have that one of the congruences 
$$
2^{u_s}\equiv 1\pmod {a_i}
$$
for $s=0,1,2,\ldots,m$ holds for our prime factor $q_i$ of $N_1$. Thus, we can write that 
$$
N_1\le \prod_{s=0}^m N_{1,s},
$$
where 
$$
N_{1,s}:=\gcd(2^{u_s}-1,2^np^m+1),\qquad {\text{\rm for}}\qquad s=0,1,2,\ldots,m.
$$ 
We now follow the argument of Lemma 3 in \cite{CLP} to bound $N_{1,s}$. Assume $n\ge (2m+2)\log p$. Put $X:=n/\log p$. Fix $s\in \{0,1,2,\ldots,m\}$, put $u:=u_s$ and consider the congruences
\begin{equation}
\label{eq:2}
2^u\equiv 1\pmod {N_{1,s}}\qquad {\text{\rm and}}\qquad 2^n p^m\equiv -1\pmod {N_{1,s}}.
\end{equation}
Look at the numbers 
$$
\{ux+ny: (x,y)\in \{0,1,\ldots,\lfloor X^{1/2}\rfloor\}.
$$
All numbers in the above set are in the interval $[0,(2m+1)nX^{1/2}]$ and there are $(\lfloor X^{1/2}\rfloor +1)^2>X$ of them. Thus, there exist $(x_1,y_1)\ne (x_2,y_2)$ such that 
\begin{eqnarray*}
|(ux_1+ny_1)-(ux_2+ny_2)| & \le &  \frac{(2m+1)nX^{1/2}}{X-1}\le \frac{(2m+2)n}{X^{1/2}}\\
& = & (2m+2){\sqrt{n\log p}}.
\end{eqnarray*}
Raising the first congruence in \eqref{eq:2} to power $x_1-x_2$ and the second to power $y_1-y_2$ and multiplying them, we get 
$$
2^{u(x_1-x_2)+n(y_1-y_2)} p^{m(y_1-y_2)}\equiv \pm 1\pmod {N_{1,s}}.
$$ 
In particular, $N_{1,s}$ divides one of 
\begin{equation}
\label{eq:3}
2^{|u(x_1-x_2)+n(y_1-y_2)|} p^{m|y_1-y_2|}\pm 1,\qquad {\text{\rm or}}\qquad 2^{|u(x_1-x_2)+n(y_1-y_2)|}\pm p^{m|y_1-y_2|},
\end{equation}
according to whether $u(x_1-x_2)+n(y_1-y_2)$ and $y_1-y_2$ have the same sign or not. Note that they are not both zero since $(x_1,y_1)\ne (x_2,y_2)$ and therefore none of the numbers shown at \eqref{eq:3} is zero. 
This shows that $N_{1,s}$ is bounded above by the maximum absolute value of the four numbers shown at \eqref{eq:3}. Since 
$$
p^{m|y_1-y_2|}\le p^{mX^{1/2}}=e^{m{\sqrt{n\log p}}},
$$
we get that 
\begin{eqnarray*}
N_{1,s} & \le  & 2^{(2m+2){\sqrt{n\log p}}} e^{m{\sqrt{n\log p}}}+1=2^{(2m+2+m/\log 2){\sqrt{n\log p}}}+1\\
& < & 2^{(3.5m+2){\sqrt{n\log p}}+1}<2^{4m{\sqrt{n\log p}}}.
\end{eqnarray*}
We thus have that 
$$
N_1\le \prod_{s=0}^m N_{1,s}<2^{4m(m+1){\sqrt{n\log p}}}.
$$
Hence, 
$$
2^np^m<2^np^m+1=N=N_0N_1< 2^{4m(m+1){\sqrt{n\log p}}} p^{4m},
$$
leading to 
$$
n<4m(m+1){\sqrt{n\log p}}+(3m/\log 2)\log p<4m(m+1){\sqrt{n\log p}}+4.5m\log p.
$$
Recalling that $X=n/\log p$, we get  
$$
X<4m(m+1)X^{1/2}+4.5m.
$$
This gives 
$$
X^{1/2}<2m(m+1)+(2m(m+1)+1)=4m(m+1)+1<5m^2,
$$
so
$$
X<25m^4.
$$
We record what we have proved so far.

\begin{lemma}
\label{lem:n}
If $N=2^np^m+1$ is Carmichael, $m$ is odd and $p>m$ is prime then $n<25m^4\log p$. 
\end{lemma}

In particular, the  numbers $|\lambda_{i}|$ appearing in congruence \eqref{eq:lambda} are all at most $mn<25m^5\log p$. Since $2^{\lambda_{i}}\equiv \pm 1\pmod {q_i}$, we get that 
$2|\lambda_{i}|$ is a multiple of the order of $2$ modulo $q_i$ which in turn is a divisor of $2^np^m$. If this order is a multiple of $p$, we then get $p\mid |\lambda_{i}|$, so 
$$
p<25m^5\log p,
$$
showing that $p< 500m^5\log m$. This finishes our proof under the assumption that the order of $2$ modulo $q_i$ is not a power of $2$ for some $q_i\mid N_1$. 
Note that if $q_i\mid N_0$, then the order of $2$ is a power of $2$. From now on, we assume that $p>500m^5\log m$. We conclude that we have proved the following lemma.

\begin{lemma}
Assume that $N=2^np^m+1$ is Carmichael for some odd $m\ge 5$, and prime $p>500m^5\log m$. Then the order of $2$ modulo $q_i$ is a power of $2$ for all prime factors $q_i$ of $N$.   
\end{lemma}

We will show that the above scenario leads to only finitely many choices for the prime $p$. We proceed through a sequence of lemmas.

\begin{lemma}
\label{lem:2}
Assume that $N=2^np^m+1$ is Carmichael for some odd $m\ge 5$ and prime $p>500m^5\log m$. Let $\gamma:=\min\{a_i: q_i\mid N\}$. 
Then $\gamma<n$. Further, there are an even number of prime factors $q_i$ of $N$ with $a_i=\gamma$. Additionally, $b\le m-1$, so divisibility \eqref{eq:102}-right cannot occur. Furthermore,  $\gamma\ge 5$. 
\end{lemma}

\begin{proof}
Let 
$$
{\mathcal A}:=\{a_i: 1\le i\le k\}.
$$
We note that $\#{\mathcal A}\ge 2$. Indeed, by Lemma \ref{lem:10}, we have that $\#{I}(b)\ge 2$. Thus, there are two distinct prime factors $q_1=2^{a_1}p^b+1$ and $q_2=2^{a_2}p^b+1$ of $N$, so ${\mathcal A}$ contains the two distinct integers $a_1,a_2$. We thus get 
$$
\gamma\le \min\{a_1,a_2\}<\max\{a_1,a_2\}\le n,
$$
therefore $\gamma\le \min\{a_1,a_2\}\le n-1$. Let $t$ be the number of prime factors  $q_i$ of $N$ which have $a_i=\gamma$. Then the congruence $q_i\equiv 1+2^{\gamma}\pmod {2^{\gamma+1}}$ holds for exactly $t$ prime factors of $N$ while for the  
remaining primes $q_j$ the congruence is $q_j\equiv 1\pmod {2^{\gamma+1}}$. Thus, since $\gamma+1\le n$, we have
\begin{eqnarray*}
1 \equiv   N\pmod {2^{\gamma+1}}\equiv (1+2^{\gamma})^t\pmod {2^{\gamma+1}}
 \equiv  1+t2^{\gamma}\pmod {2^{\gamma+1}} 
\end{eqnarray*}
showing that $t$ is even. In particular, $t\ge 2$. Thus, there are two distinct prime factors $q_1=2^\gamma p^{b_1}+1$ and $q_2=2^{\gamma}p^{b_2}+1$ of $N$, so $b_1\ne b_2$. Hence, $b\le \min\{b_1,b_2\}<\max\{b_1,b_2\}\le m$, 
so $b\le m-1$, which shows that divisibility \eqref{eq:102}-right cannot occur.

Assume next that $\gamma\le 4$. Then there exists a prime factor $q_i$ of $N$ of the form $2^{\gamma} p^j+1$ for some $j\in \{1,2,\ldots,m\}$ (since $t\ge 2$ so $q_i=2^{\gamma}+1$ having $(a_i,b_i)=(\gamma,0)$ cannot be the only prime factor 
of $N$ with $a_i=\gamma$). In particular,  the order of $2$ modulo $2^{\gamma}p^j+1$ is $2^{\lambda}$ for some $\lambda\le 4$. However, $2^{2^{\lambda}}-1$ is only divisible by Fermat primes 
for all $\lambda\le 4$, a contradiction. This shows that $\gamma\ge 5$ in this case.
\end{proof}

We next look at the arithmetic of the exponents $a_1,\ldots,a_k$ and $n$.

\begin{lemma}
\label{lem:12}
If $N=2^np^m+1$ is Carmichael, $m\ge 5$ is odd, $p>500m^5\log m$ is prime, then $n=2^{n_0}n_1$, where $n_1$ is odd and $2^{n_0}>(\log p)/m^3$. Additionally,  for every prime factor $q_i=2^{a_i}p^{b_i}+1$ of $N$, we have $a_i=2^{\beta_i}\lambda_i$ where $\lambda_i$ is odd and 
$2^{\beta_i}>(\log p)/m^3$. 
\end{lemma}

\begin{proof}
Assume $q_i=2^{a_i}p^{b_i}+1$ is a prime factor of $N$ with $b_i>0$. Let ${\text{\rm ord}}_2(q_i)=2^{\omega_i}$. We then have 
$$
2^{2^{\omega_i-1}}\equiv -1\pmod {q_i}.
$$
Thus, 
\begin{equation}
\label{e:10}
2^{2^{\omega_i-1}}+1\ge 2^{a_i}p^{b_i}+1>p+1,\quad {\text{\rm therefore}}\quad 2^{\omega_i-1}\ge \frac{\log p}{\log 2}>\log p.
\end{equation}
We shall exploit the above inequality for various prime factors $q_i$ of $N$ with $b_i>0$. We write $a_i=2^{\beta_i}\gamma_i$ for some  nonzero negative integers $\beta_i$ and odd integers $\gamma_i$. 
Returning to congruences \eqref{eq:lambda}, we get that 
$$
2^{\lambda_i}\equiv \varepsilon_i\pmod {q_i},
$$
where $\lambda_i=nc_i-a_im_i$ and $\varepsilon_i\in \{\pm 1\}$. It follows that 
\begin{equation}
\label{eq:100}
nc_i-a_im_i=2^{\omega_i-1}\delta_i,\quad {\text{\rm holds~with~some~nonzero~integer}}\quad \delta_i.
\end{equation}
The fact that $\delta_i$ is nonzero follows from the fact that $2^np^m$ and $2^{a_i}p^{b_i}$ are multiplicatively independent for all prime factors $q_i=2^{a_i}p^{b_i}+1$ with $b_i>0$ of $N$. 
We now calculate $2$-adic valuations in \eqref{eq:100} above. Write $n=2^{n_0}n_1$ with $n_1$ odd. Assume first that 
\begin{equation}
\label{eq:11}
2^{n_0}\ge\frac{\log p}{m^3}. 
\end{equation}
Note that since $m_i\mid m$ is odd, it follows that $\nu_2(a_im_i)=\nu_2(a_i)=\beta_i$. Thus, 
we get that 
$$
\beta_i=\nu_2(a_im_i)\ge \min\{\nu_2(nc_i),\nu_2(2^{\omega_i-1}\delta_i)\}\ge \min\{n_0,\omega_i-1\},
$$
therefore 
$$
2^{\beta_i}\ge \min\{2^{n_0},2^{\omega_i-1}\}\ge \frac{\log p}{m^3}.
$$
So, from now on we shall assume that \eqref{eq:11} fails. We take the $t\ge 2$ primes of the form $q_i=2^{\gamma}p^{b_i}+1$ which are guaranteed by Lemma \ref{lem:2}. Assume that there are two of them say $q_i=2^{\gamma}p^{b_i}+1$ and $q_j=2^{\gamma}p^{b_j}+1$ with $b_ib_j> 0$. We then get that 
$$
nc_i-\gamma m_i=2^{\omega_i-1}\delta_i\qquad {\text{\rm and}}\qquad nc_j-\gamma m_j=2^{\omega_j-1}\delta_j.
$$
Eliminating $\gamma$ from the above two equations we have 
$$
n(m_jc_i-m_ic_j)=2^{\min\{\omega_i-1,\omega_j-1\}}\Lambda
$$
for some integer $\Lambda$. Note that $m_ic_j\ne m_jc_i$. Indeed, if they are equal then $m_i/c_i=m_j/c_j$. But $m_i/c_i=m/b_i$ and $m_j/c_j=m/b_j$. So, these fractions being equal entails $b_i=b_j$, so $q_i=q_j$ 
which is false. Now 
$$
|m_ic_j-m_jc_i|<m^2,
$$
which shows that $\nu_2(|m_ic_j-m_jc_i|)<2\log m/\log 2$. This shows that 
$$
2^{n_0}> \frac{2^{\min\{\omega_i-1,\omega_j-1\}}}{m^2}>\frac{\log p}{m^2},
$$
contradicting \eqref{eq:11}. Thus, we have to assume that $t=2$, and up to relabelling the primes we have $q_1=2^{\gamma}+1$, $\gamma=2^{\beta}$, and $q_2=2^{\gamma}p^{b_2}+1$ for some $b_2>0$. We write again 
\begin{equation}
\label{eq:12}
nc_2-\gamma m_2=2^{\omega_2-1}\delta_2.
\end{equation}
Assume next that $b_2$ is odd. Then $c_2\mid b_2$ is also odd, so 
$$
\nu_2(nc_2)=\nu_2(n)=\nu_2(\gamma m_2)=\nu_2(\gamma)=\beta.
$$ 
In particular, $n_0=\beta$. Now 
$$
N=2^np^m+1=2^{\gamma n_1} p^m+1\equiv (-1)^{n_1} p^m+1\pmod {q_1}\equiv p^m-1\pmod {q_1}.
$$
Thus, $p$ has order dividing $m$ modulo $q_1$. Since $q_1$ is a Fermat prime, the order of $p$ modulo $q_1$ must be a power of $2$. Since this order divides $m$ which is odd, it must be $1$. This shows that $p\equiv 1\pmod {q_1}$. Now 
$$
q_2=2^{\gamma}p^{b_2}+1\equiv (-1) (1)^{b_2}+1\equiv 0\pmod {q_1},
$$
which shows that $q_1\mid q_2$, which is of course impossible since $q_1$ and $q_2$ are distinct primes. Thus, $b_2$ is even. Let $a_2:=\min\{a_{s}: s\ge 3\}$. Note that $a_2$ exists since 
$k=\omega(N)\ge 3$. Reducing relation \eqref{eq:0} modulo $2^{a_2}$ we get that
$$
1\equiv (2^{\gamma}p^{b_2}+1)(2^{\gamma}+1)\equiv 1+2^{\gamma}(p^{b_2}+1)+2^{2\gamma}p^{b_2}\pmod {2^{a_2}},
$$
which shows that $2^{\min\{a_2,2\gamma\}}\mid 2^{\gamma}(p^{b_2}+1)$. Since $b_2$ is even, we get that $\min\{a_2,2\gamma\}=\gamma+1$. Since $\gamma\ge 8$, we get that $a_2=\gamma+1$. In particular, 
$a_2$ is odd. Let $q_3=2^{\gamma+1}p^{b_3}+1$. Clearly, $b_3\ne 0$ since otherwise $q_3$ should be a Fermat prime which is false since $\gamma+1\ge 9$ is odd. Thus, $b_3\ge 1$. Note also that $b_3$ is odd since
if $b_3$ is even then $q_3=2^{\gamma+1}(p^2)^{b_3/2}+1\equiv 0\pmod 3$, so $q_3$ is not prime. We now write 
\begin{equation}
\label{eq:13}
nc_3-(\gamma+1)m_3=2^{\omega_3-1}\delta_3.
\end{equation}
Note that $(\gamma+1)m_3$ is odd showing that $nc_3$ is also odd. In particular, $n$ is odd. Since by \eqref{eq:12} we have that $\nu_2(nc_2)=\nu_2(\gamma m_2)=\nu_2(\gamma)=\beta$, we get that 
$\beta=\nu_2(\gamma)=\nu_2(c_2)\le \log m/\log 2$. Thus, $\gamma=2^{\beta}\le m-1$. We now eliminate the terms containing $\gamma$ in \eqref{eq:12} and \eqref{eq:13} getting 
\begin{equation}
\label{eq:14}
n(c_2 m_3(\gamma+1)-c_3m_2\gamma)=2^{\min\{\omega_2-1,\omega_3-1\}}\Lambda',
\end{equation}
for some integer $\Lambda'$. Let us prove that the factor multiplying $n$ above is nonzero. If it were zero, we would get
$$
\frac{c_2(\gamma+1)}{m_2}=\frac{c_3\gamma}{m_3}.
$$
Since $c_2/m_2=b_2/m$ and $c_3/m_3=b_3/m$, we get that $(\gamma+1)b_2=\gamma b_3$, so $b_2=\gamma d$ and $b_3=(\gamma+1)d$ for some integer $d$. Thus, 
$$
q_2=2^{\gamma} p^{b_2}+1=(2p^d)^{\gamma}+1,\qquad q_3=2^{\gamma+1} p^{b_3}+1=(2p^d)^{\gamma+1}+1,
$$
and we see that $q_3$ is divisible by $2p^d+1$ (since $\gamma+1>1$ is odd), so in particular it is not prime. Thus, in \eqref{eq:13}, the factor multiplying $n$ is not zero. Its size is at most
$$
|c_2m_3(\gamma+1)-c_3m_2\gamma|\le \max\{c_2,c_3\}\max\{m_2,m_3\}(\gamma+1)\le m^3,
$$
which shows that 
$$
2^{n_0}> \frac{2^{\min\{\gamma_2-1,\gamma_3-1\}}}{m^3}\ge \frac{\log p}{m^3},
$$
so in fact \eqref{eq:11} holds. The lemma is therefore proved.  
\end{proof}

Let 
$$
\beta':=\min\{\nu_2(n),\nu_2(a_1),\ldots,\nu_2(a_k)\}.
$$
Lemma \ref{lem:12} shows that for $p>500m^5\log m$ prime, we have $2^{\beta'}\ge (\log p)/m^3$. Further, write 
$$
n=2^{\beta'}n_1',\qquad a_i=2^{\beta'}\lambda_i',\qquad 1\le i\le k.
$$
We then have that 
$$
q_i=2^{a_i}p^{b_i}+1\ge 2^{a_i}\ge 2^{2^{\beta'}\lambda_i'}\ge 2^{(\log p)(\lambda_i'/m^3)},\qquad 1\le i\le k.
$$
Thus,  by Lemma \ref{lem:n}, we have that 
\begin{equation}
\label{eq:110}
N=2^np^m+1<2^{n+m(\log p)/\log 2}+1<2^{(25m^4+2m)\log p}<2^{26m^4\log p},
\end{equation}
while also 
\begin{equation}
\label{eq:111}
N=\prod_{i=1}^k q_i>\prod_{i=1}^k 2^{(\log p)(\lambda_i'/m^3)}=2^{((\log p)/m^3)\sum_{i=1}^k \lambda_i'}.
\end{equation}
The combination of \eqref{eq:110} and \eqref{eq:111} show that 
$$
\sum_{i=1}^k\lambda_i'<26m^7,
$$
which implies that the inequalities $k\le 26m^7$ and $\lambda_i'\le 26m^7$ hold for all indices $i=1,\ldots,k$. Further, 
$$
24m^4\log p\ge n=2^{\beta'} n_1'\ge (\log p)(n_1'/m^3),
$$
so  $n_1'\le 24m^7$. Also $b_i\le m$ for $i=1,\ldots,k$. Since $m$ is fixed, and we aim to show that there are only finitely many possibilities for the prime $p>500m^5\log m$, we may assume that there are infinitely many possibilities for $p$ and therefore there are infinitely many possibilities for $p$ for which 
$$
n_1',~k,~\lambda_1',\ldots,\lambda_k',b_1,\ldots,b_k
$$ 
are all fixed. We thus get that the equation 
$$
P(X,Y)=0,
$$
where 
\begin{equation}
\label{eq:P}
P(X,Y)=(X^{n_1'}Y^m+1)-\prod_{i=1}^k (X^{\lambda_i'}Y^{b_i}+1)
\end{equation}
has infinitely many positive integer solutions $(X,Y)=(2^{2^{\beta'}},p_{\beta'})$, where $p_{\beta'}$ is prime.

We record this as a lemma. 

\begin{lemma}
\label{lem:CZ}
If there exist infinitely many Carmichael numbers of the form $2^np^m+1$ for some fixed odd number $m\ge 5$ and a prime $p>500m^5\log m$, then there exist fixed 
$$
k,n_1',\lambda_1,\ldots,\lambda_k,b_1,\ldots,b_k
$$
satisfying $\max\{k,\lambda_1',\ldots,\lambda_k'\} <26m^7,~n_1'\le 24m^7, \max\{b_1,\ldots,b_k\}\le m$ and $(\lambda_i',b_i)\ne (\lambda_j',b_j)$ 
for $i\ne j$ in $\{1,\ldots,k\}$ such that the equation
$$
P(X,Y)=0,
$$
where $P(X,Y)$ is given by \eqref{eq:P} has infinitely many positive integer solutions of the form 
$$
(X,Y)=(2^{2^{\beta'}},p_{\beta'}),
$$ 
where $\beta'$ is a positive integer and $p_{\beta'}$ is a prime. 
In addition, also the quantities 
$$
q_i=X^{\lambda_i'}Y^{b_i}+1=2^{2^{\beta'}\lambda_i'}p_{\beta'}^{b_i}+1\quad {\text{\it are distinct primes for }}\quad i=1,\ldots,k.
$$  
\end{lemma}

Theorem 2 in \cite{CZ} characterizes such instances. Instead of indicating the full statement of that theorem, we follow its proof on page 324--327 in \cite{CZ} for our particular instance. 
Then we will obtain a contradiction by deriving various conclusions about the involved polynomials which cannot all be fulfilled for our situation. Let's get to work. 

By (2) on Page 324 in \cite{CZ}, 
there are polynomials $f(X),g(X)\in {\mathbb Q}[X]$ such that for each convenient positive integer $\beta'$ as above, we have 
$$
2^{2^{\beta'}}=f(x_{\beta'})\qquad {\text{\rm and}}\qquad p_{\beta'}=g(x_{\beta'})
$$
for some $x_{\beta'}\in {\mathbb Q}$. We start with the polynomial $f(X)\in {\mathbb Q}[X]$. Since there are infinitely many  rational numbers $x_{\beta'}$ such that $f(x_{\beta'})$ is a power of $2$, it follows that
$x_{\beta'}$ have bounded denominators by some number $K$ depending only on $f(X)$. Next,  $f(X)$ is associated to a power of a linear polynomial (indeed, if $f(X)$ has two or more distinct roots then the largest prime factor of the numerator of $f(s)$ where 
$s$ runs through rational numbers with denominator bounded by $K$ tends to infinity as the absolute values of such rational numbers $s$ tend to infinity). Up to a linear change in $x_{\beta'}$, we may assume that $f(X)=rX^d$ for some rational number $r$ and integer $d$. If $d$ is even, then the equation
\begin{equation}
\label{eq:201}
2^{2^{\beta'}}=rx_{\beta'}^d,\qquad x_{\beta'}\in {\mathbb Q}
\end{equation}
with $\beta'\ge 1$ implies that $r=r_1^2$ is a perfect square and $2^{2^{\beta'-1}}=r_1x_{\beta}^{d/2}$. Proceeding in this way, we deduce that if $d=2^{d_0}d_1$, where $d_0\ge 0$ is an integer and $d_1\ge 1$ is odd, then 
we may assume that $r=r_0^{2^{d_0}}$ for some rational number $r_0$ and so equation \eqref{eq:201} implies that 
$$
2^{2^{\beta'-d_0}}=r_0x_{\beta'}^{d_1}\qquad {\text{\rm holds whenever}}\qquad \beta'\ge d_0.
$$
From the above we see that every prime number except for $2$ appears in the prime  factorization of $r_0$ at exponent which is a multiple of $d_1$, thus, up to a linear change of variable in  $x_{\beta'}$, we may assume that 
$r_0=2^{\beta_0}$ for some integer $\beta_0$. We thus get that 
$$
x_{\beta'}^{d_1}=2^{2^{\beta'-d_0}-2^{\beta_0}},\qquad {\text{\rm therefore}}\qquad x_{\beta'}=2^{\frac{2^{\beta'-d_0}-2^{\beta_0}}{d_1}}.
$$
In the above, all the roots are the positive real determinations.  Thus, 
$$
x_{\beta}=\alpha 2^{\frac{2^{\beta'}}{d}},\qquad {\text{\rm where}}\qquad \alpha=2^{-\frac{2^{\beta_0}}{d_1}}.
$$
The expression $x_{\beta'}$ is a rational number whenever $\beta'$ is a positive integer such that $\beta'-d_0\equiv \beta_0\pmod {\ell_2(d_1)}$, where we write $\ell_2(d_1)$ for the order of $2$ modulo the odd integer $d_1$. 

In particular, letting $Z:=2^{\frac{2^{\beta'}}{d}}$, we have $2^{2^{\beta'}}=Z^d$ and 
$$
p_{\beta'}=g(x_{\beta})=g(\alpha Z)\in {\mathbb Q}(\alpha)[Z].
$$ 
Writing 
$$
p(Z):=g(\alpha Z)\in {\mathbb L}[Z],
$$
where ${\mathbb L}:={\mathbb Q}(2^{1/d_1})$, we get that the relation 
\begin{equation}
\label{eq:112}
P(Z^d,p(Z))=0
\end{equation}
holds for infinitely many values of $Z$, namely for infinitely many values of $Z$ of the form $Z=2^{\frac{2^{\beta'}}{d}}$. In particular, the above relation signals a polynomial equation with infinitely many roots, so 
\begin{equation}
\label{eq:Pp}
P(Z^d,p(Z))=0
\end{equation}
holds identically for all $Z$. Let us learn more things about the polynomial $p(Z)$. Let us revisit Lemma \ref{lem:10}. Recasted in our new variables, Lemma \ref{lem:10} and Lemma \ref{lem:2} show that 
one of the following holds:
\begin{itemize}
\item[(i)] If $b=0$, let $I(b)=I(0)$ be the set of $i\in \{1,\ldots,k\}$ such that $b_i=0$ and assume that $a_i=2^{\beta'}\lambda_i'$, where $\lambda_i'\le 26m^7$ are fixed powers of $2$. Then
$$
p(Z)\mid \prod_{i\in I(0)} (Z^{d\lambda_i'}+1)-1.
$$
\item[(ii)] Assume that $1\le b\le m-1$. Then 
$$
p(Z)\mid \prod_{i\in I(b)} Z^{d\lambda_i'}.
$$
\end{itemize}
In all cases, $\#I(b)\ge 2$ and the above divisibility is to be interpreted in ${\mathbb C}[Z]$. Both instances (i) and (ii) show that the roots of $p(Z)$ are algebraic integers as $p(Z)$ divides some 
monic polynomial with integer coefficients. Next, let us compare degrees and leading terms in \eqref{eq:Pp}. 
In the identity 
$$
Z^{dn_1'}p(Z)^{m}+1=\prod_{i=1}^k (Z^{d\lambda_i'} p(Z)^{b_i}+1),
$$
equating the degrees of the polynomials from the left and right give the equation 
\begin{equation}
\label{eq:2000}
{\text{\rm deg}}(p)(m-\sum_{i=1}^{k} b_i)=d(\sum_{i=1}^k \lambda_i'-n_1').
\end{equation}
Letting $\lambda$ be the leading coefficient of $p(Z)$, we get by equating the leading coefficients of the left and right hand sides above 
\begin{equation}
\label{eq:2001}
\lambda^{m-\sum_{i=1}^k b_i}=1.
\end{equation}
If $m-\sum_{i=1}^k b_i=0$, then \eqref{eq:2000} shows that $m=\sum_{i=1}^k b_i$ and $n_1'=\sum_{i=1}^k \lambda_i'$. However, this is impossible as by multiplication with any suitable  $2^{\beta'}$ leads to
$n=\sum_{i=1}^k a_i$, so in equation \eqref{eq:0} we have 
$$
(q_1-1)(q_2-1)\cdots (q_k-1)+1=q_1\cdots q_k(=N),
$$
which is false as the left--hand side is smaller than the right--hand side for all $k\ge 3$ and any integers $3\le q_1<\cdots<q_k$. This shows that in equation \eqref{eq:2001}, the exponent 
$m-\sum_{i=1}^k b_i$ is nonzero, so $\lambda$ is a root of unity. In particular, since the leading coefficient of $p(Z)$ is a root of unity and all its roots are algebraic integers, 
we get that the coefficients of $p(Z)$ are algebraic integers. Next, let $i\in \{1,\ldots,k\}$ be such that $b_i\ge 1$. We look at the polynomial 
$$
q_i(Z)=Z^{d\lambda_i'}p(Z)^{b_i}+1.
$$
The roots of $\zeta$ of $q_i(Z)$ are roots of unity. To see that, let $\zeta$ be such a root. Then $\zeta$ satisfies the relations
$$
\zeta^{d\lambda_i'} p(\zeta)^{b_i}=-1\quad {\text{\rm and also}}\quad \zeta^{dn_1'}p(\zeta)^{m}=-1.
$$
Eliminating $p(\zeta)$ from the above relations, we get that $\zeta^{d(n_1c_i-\lambda_i'm_i)}=\varepsilon_i$, where $c_i$ and $m_i$ have the same meaning as right before display  \eqref{eq:lambda} (that is, we have the relations $d_i=\gcd(m,b_i),~b_i=d_ic_i,~m=d_im_i$).  Moreover $d(n_1c_i-\lambda_i'm_i)$ is nonzero since $2^{n}p^m$ and $2^{a_i}p^{b_i}$ are multiplicatively independent.  This shows that $\zeta$ is a root of unity.
Since $\lambda$ is also a root of unity, it follows that $q_i(Z)\in {\mathbb K}[Z]$, where ${\mathbb K}:={\mathbb Q}(e^{2\pi i/M})$ for some positive integer $M$ is some cyclotomic extension of 
${\mathbb Q}$. However, 
\begin{equation}
\label{eq:5000}
q_i(Z)=Z^{dn_1'} p(Z)^{b_i}+1=(\alpha^{-d})^{n_1'} (\alpha Z)^{dn_1'}g(\alpha Z)^{b_i}+1=\sum_{i=0}^w c_i \alpha^i Z^i,
\end{equation}
where $c_0,\ldots,c_w$ are rational numbers since $\alpha^{-d}\in {\mathbb Q}$. We thus get that $q_i(Z)\in ({\mathbb K}\cap {\mathbb L})[z]$. It is easy to see that ${\mathbb K}\cap {\mathbb L}={\mathbb Q}$. Indeed, ${\mathbb K}\cap {\mathbb L}$ is an abelian extension of ${\mathbb Q}$ (as a subfield of a cyclotomic field), and also a subfield of ${\mathbb L}={\mathbb Q}(2^{1/d_1})$. However, since $d_1$ is odd, the only proper subfields of ${\mathbb Q}(2^{1/d_1})$ are ${\mathbb Q}(2^{1/d_2})$ for divisors $d_2$ of $d_1$ (see Lemma 5.1 part 2 in \cite{BMM}) which are also odd,  and the only such abelian extension is for $d_2=1$. Thus, $q_i(Z)\in {\mathbb Q}[Z]$. Since its leading coefficient $\lambda^{b_i}$ is a root of unity and rational, 
it follows that the leading coefficient is $\pm 1$. Since the coefficients of $p(Z)$ are algebraic integers, so are the coefficients of $q_i(Z)$. In particular, $q_i(Z)$ has integer coefficients and 
leading coefficient $\pm 1$, so it follows that $q_i(Z)\in {\mathbb Z}[Z]$ for $i=1,\ldots,k$. Since $q_i(Z)$ is a prime for infinitely many values specializations of $Z$ (in numbers of the form $2^{\frac{2^{\beta'}}{d}}$ for suitable positive integers $\beta'$), it follows that it has leading coefficient $1$. 
Next, we note that either $d_1=1$ or if not, then since $c_i\alpha^{i}\in {\mathbb Q}$ for all $i=0,1,\ldots,w$ in \eqref{eq:5000}, it follows that $c_i=0$ for all $i=0,1,\ldots,w$ such that
$i\not\equiv 0\pmod {d_1}$. It now follows that in fact 
$$
Z^{dn_1'} p(Z)^{b_i}+1=q_i(Z)=q_{i,1}(Z^{d_1}),
$$
where 
$$
q_{i,1}(Z):=\sum_{\substack{0\le i\le w\\ i\equiv 0\pmod {d_1}}} c_{i}Z^{i/d_1}\in {\mathbb Z}[Z]. 
$$
Hence, it follows that 
$$
p(Z)^{b_i}=\frac{q_{i,1}(Z^{d_1})-1}{(Z^{d_1})^{n_1'}}:=r(Z^{d_1}),
$$
where $r\in {\mathbb Z}[Z]$ is monic. By looking at the multiplicities of the irreducible factors of $r(Z)$, we conclude that $r(Z)=p_1(Z)^{b_i}$, so 
$$
p(Z)=p_1(Z^{d_1}),
$$
for some polynomial $p_1(Z)\in {\mathbb Q}[Z]$. It is easy to see that $p_1(Z)$ also has leading coefficient $1$, integer coefficients since its coefficients are both 
algebraic integers (by the Vi\'ete formulas) and rational numbers, and it is irreducible in ${\mathbb Z}[Z]$ since it takes on infinitely many prime values when specialized  in suitable for $Z$ of the form $2^{2^{\beta'-d_0}}$ for infinitely many positive integers 
$\beta'$.  Thus, in \eqref{eq:112}, we may replace $d$ by $d_0'=2^{d_0}$ and $P(Z)$ by $P_1(Z)$ and the analog \eqref{eq:112} namely the relation
$$
Z^{d_0'n_1'}P_1(Z)^m+1=\prod_{i=1}(Z^{d_0'\lambda_i'} P_1(Z)^{b_i}+1)
$$
holds identically. Next, since $q_i(Z)\in {\mathbb Z}[Z]$ have roots which are roots of unity and are monic and irreducible, it follows  that
$$
q_i(Z)=\Phi_{M_i}(Z)
$$
for some integer ${M_i}$ and all $i=1,\ldots,k$, where $\Phi_{M}(z)$ is the $M$th cyclotomic polynomial. We are finally ready to reach a contradiction. Assume first that $b_i\ge 2$ for some $i\in \{1,\ldots,k\}$. Then 
$$
Z^{d_0'\lambda_i'}p_1(Z)^{b_i}=\Phi_{M_i}(Z)-1.
$$
We may change $Z$ to $Z^2$ above (if $d_0'\lambda_i'=1$) and get 
$$
Z^{2d_0'\lambda_i'}p_1(Z^2)^{b_i}=\Phi_{M_i}(Z^2)-1.
$$
In the right--hand side above, we have that $\Phi_{M_i}(Z^2)=\Phi_{2M_i}(Z)$ if $M_i$ is even and $\Phi_{M_i}(Z^2)=\Phi_{M_i}(Z)\Phi_{2M_i}(Z)$ if $M_i$ is odd. The polynomial on the left has all its roots of multiplicity at least $2$. 
This means that all roots of $\Phi_{M_i}(Z^2)-1$ are also multiple, so they must be among the roots of the derivative of $\Phi_{M_i}(Z^2)$. But the roots of the derivative $H'(Z)$ of a polynomial with complex coefficients $H(Z)$ are 
inside the convex hull of the roots of $H(Z)$. Since $H(Z)=\Phi_{M_i}(Z^2)$ has   $2\phi(M_i)$ roots on the unit circle which are all simple, it follows that the roots of its derivative are strictly inside the unit circle. 
This contradicts the fact that the roots of $p_1(Z^2)$ are algebraic integers via Kronecker's theorem. Thus, the situation when $b_i\ge 2$ for some $i=1,\ldots,k$ is not possible. 

So, it remains to treat the situation when $b_i\in \{0,1\}$ for $i=1,\ldots,k$. Then $\{1,\ldots,k\}=I(0)\cup I(1)$, where 
$$
I(c):=\{1\le i\le k, q_i(Z)=Z^{d_0'\lambda_i} p_1(Z)^c+1\}.
$$ 
As we noticed, for $c=1$ we have that $q_i(Z)$ are all cyclotomic polynomials. This is true for $c=0$ as well since in this case $q_i(Z)=Z^{d_0\lambda_i'}+1=\Phi_{2d_0'\lambda_i'}(Z)$ as $d_0'$ and $\lambda_i'$ are powers of $2$. 
Thus, 
$$
Z^{d_0'n_1'}p_1(Z)^m=\prod_{i=1}^k \Phi_{M_i}(Z)-1,
$$
for distinct integers $M_1,\ldots,M_k$. Up to replacing $Z$ by $Z^2$ (if $d_0'n_1'=1$), we may assume that the left--hand side is a polynomial all whose roots are multiple. It then follows that all roots of the polynomial 
on the right--hand side are also multiple. Now we get the same contradiction as in  the case when $b_i\ge 2$ for some $i=1,\ldots,k$, namely that the roots of the derivative of the polynomial on the right--hand side are strictly inside the unit circle, contradicting the fact that the roots of $p_1(Z)$ are algebraic integers. 

This contradiction shows that the  conclusion of the Lemma \ref{lem:CZ} cannot hold, therefore the hypothesis of that lemma cannot hold either. Hence, there are only finitely many Carmichael numbers of the form $N=2^np^m+1$ for fixed $m$ odd and primes $p>500 m^5\log m$, but this part of the proof is ineffective in that we don't have an upper bound for the largest such $p$ depending on $m$.

\section{Acknowledgements}

We thank Yuri Bilu for providing a reference. During the preparation of this paper F. L. was partly supported by the 2024 ERC Synergy Project {\it DinAMiCs}.

\end{document}